\documentclass[12pt]{article}
\usepackage[centertags]{amsmath}
\usepackage{amsfonts}
\usepackage{amssymb}
\usepackage{amsthm}
\usepackage{amsmath,mathrsfs}
\usepackage{amsmath,amscd}
\title{\textbf{Relaxed $\check{C}$ech Cohomology, Emeralds Over Topological Spaces and The Kontsevich Integral}}
\author{Renaud Gauthier \footnote{rg.mathematics@gmail.com} \\ \\Lycee Albert Camus}
\theoremstyle{definition}
\newtheorem{deco}{Definition}[section]
\newtheorem{D}[deco]{Definition}
\newtheorem{Dwithsheaf}[deco]{Definition}
\newtheorem{emerald}[deco]{Definition}
\newtheorem{logdiffsheaf}[deco]{Example}
\newtheorem{Omsheaf}[deco]{Proposition}
\newtheorem{relCech}{Definition}[section]
\newtheorem{PartLXN}{Definition}[section]
\newcommand{\beq}{\begin{equation}}
\newcommand{\eeq}{\end{equation}}
\newcommand{\dlog}{\text{dlog}}
\begin{document}
\maketitle
\begin{abstract}
We introduce families of decorations of a same topological space, as well as a family of sheaves over such decorated spaces. Making those families a directed system leads to the concept of emerald over a space. For the configuration space $X_N$ of $N$ points in the plane, connecting points of the plane with chords is a decoration and the sheaf of log differentials over such spaces forms an emerald. We introduce a relaxed form of $\check{C}$ech cohomology whereby intersections are defined up to equivalence. Two disjoint open sets of $X_N$ whose respective points are connected by a chord is one instance of intersection up to equivalence. One paradigm example of such a formalism is provided by the Kontsevich integral.
\end{abstract}
\newpage

\section{Introduction}
The aim of this short paper is to introduce a few concepts that the author deems to be of relevance in the study of singular knots. Vassiliev was among the first to grasp the importance of seeing knot invariants as arising from the desingularization of knots with transversal intersections \cite{V}. One important case of such a knot invariant is the Kontsevich integral \cite{K} which is valued in the completion of a graded algebra $\mathcal{A}$ of geometric objects such as links with horizontal chords ending on their components. The configuration space $X_N$ of $N$ points in the complex plane comes up naturally in the study of such graphs with chords. We are therefore led to consider spaces with decorations on their points such as chords stretching from one point to another for instance. If we have a family $\mathcal{D}$ of such decorations of a same space $X$, the next thing to consider is the introduction of sheaves $\mathcal{F}_d$ over each decorated space $d(X)$ of $X$, $d \in \mathcal{D}$. We refer to such a construction as a decorated fan over $X$. If the collection of decorations forms a directed system and the corresponding sheaves $\mathcal{F}_d$ are correspondingly organized as a directed system, we regard such a system of sheaves as originating from a same sheaf $\mathcal{F}$, and the triple $(X, \mathcal{D}, \mathcal{F})$ we call an emerald over $X$. A paradigm example is provided by the sheaf of log differentials $\dlog(z_i-z_j)$ over $X_N$ on which we choose a directed system of decorations to be given by chords between points in the plane. Chords between components of a given link arise in the study of singular links and correspond to points which are identified on the singular link. Thus we are led to define a relaxed form of $\check{C}$ech cohomology whereby intersections are defined up to some equivalence. In the case of chords between points, distinct open sets whose points are connected by a chord are considered to be intersecting up to chords, and we are thus regarding the collection of log differentials as one forms in relaxed $\check{C}$ech cohomology. This formalism can easily be applied to the Kontsevich integral.\\

Rather than repeat the formalism of \cite{RG} on which this work is based, the reader is referred to that paper for definitions regarding the configuration space of $N$ points in the plane, chord diagrams, the Kontsevich integral and chord diagram valued log differentials.\\

In section 2 we define decorations $d$, decorated fans $(X, \mathcal{D}, \{\mathcal{F}_d\})$ as well as emeralds $(X, \mathcal{D}, \mathcal{F})$. In section 3 we define relaxed $\check{C}$ech cohomology, which will enable us to view log differentials as one forms over decorated spaces. In section 4 we apply this formalism to the Kontsevich integral.

\section{Emeralds over Topological Spaces}
\begin{deco}
For $X$ a topological space, $d$ a decoration of the space $X$, we let $d(X)=\{dx \: | \: x \in X \}$. There is a natural projection $d(X) \xrightarrow{\pi_d} X$ which strips every point $d(x)$ of $d(X)$ of its decoration to yield back the point $x \in X$. We topologize $d(X)$ using the topology of $X$.
\end{deco}
\begin{D}
Let $\mathcal{D}$ be a collection of decorations $d$ of a same topological space $X$. Then the pair $(X, \mathcal{D})$ is called a decorated fan over $X$.
\end{D}
\beq
\setlength{\unitlength}{0.5cm}
\begin{picture}(8,7)
\put(0,6){$d(X)$}
\put(5,6){$d'(X)$}
\multiput(8,6.3)(0.5,0){3}{\circle*{0.15}}
\linethickness{0.3mm}
\put(2,5){\vector(4,-3){4}}
\put(6,5){\vector(1,-3){0.9}}
\put(2,3){$\pi_d$}
\put(7,4){$\pi_{d'}$}
\put(6.8,1){$X$}
\end{picture}
\eeq
\begin{Dwithsheaf}
Let $X$ be a topological space, $\mathcal{D}$ a collection of decorations $d$ of the points of $X$, $\mathbf{F}=\{\mathcal{F}_d \: | \: d \in \mathcal{D} \}$ a collection of sheaves $\mathcal{F}_d$ over each decorated space $d(X)$. We call the triple $(X, \mathcal{D}, \mathbf{F})$ a decorated fan of $X$ with sheaves.
\end{Dwithsheaf}
\beq
\setlength{\unitlength}{0.5cm}
\begin{picture}(8,10)
\put(0,6){$d(X)$}
\put(5,6){$d'(X)$}
\put(1,9){\vector(0,-1){2}}
\put(1,10){$\mathcal{F}_d$}
\put(5.5,9){\vector(0,-1){2}}
\put(5,10){$\mathcal{F}_{d'}$}
\multiput(8,10.3)(0.5,0){3}{\circle*{0.15}}
\multiput(8,6.3)(0.5,0){3}{\circle*{0.15}}
\linethickness{0.3mm}
\put(2,5){\vector(4,-3){4}}
\put(6,5){\vector(1,-3){0.9}}
\put(2,3){$\pi_d$}
\put(7,4){$\pi_{d'}$}
\put(6.8,1){$X$}
\end{picture}
\eeq
\begin{emerald}
Let $X$ be a topological space, $\mathcal{D}$ a directed set of decorations of $X$, $(X, \mathcal{D}, \mathbf{F})$ a decorated fan with sheaves over $X$. Suppose the map on sheaves $\mathcal{F}_d$ over each decorated space $d(X)$ induced by the partial order on such decorated spaces is a morphism of sheaves that amounts to building a sheaf $\mathcal{F}_{d'}$ from its predecessor $\mathcal{F}_d$, $d \leq d'$, and if in addition $d_0$ is a smallest decoration for the partial order on $\mathcal{D}$, then writing $\mathbf{F}=\mathcal{F}=\mathcal{F}_{d_0}$, we call $(X, \mathcal{D}, \mathcal{F})$ an emerald over $X$. We regard $\mathcal{F}$ as being $\mathcal{F}_d$ over each $d(X)$, $d \in \mathcal{D}$.
\end{emerald}
\beq
\setlength{\unitlength}{0.5cm}
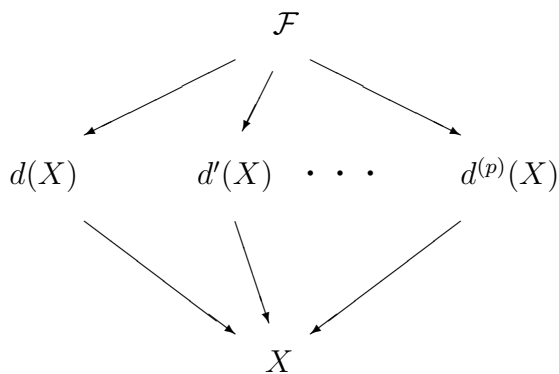
\begin{picture}(14,10)
\put(0,6){$d(X)$}
\put(5,6){$d'(X)$}
\put(12,6){$d^{(p)}(X)$}
\put(6,9.3){\vector(-2,-1){4}}
\put(7,9){\vector(-1,-2){0.8}}
\put(8,9.3){\vector(2,-1){4}}
\put(7,10){$\mathcal{F}$}
\multiput(8,6.3)(0.8,0){3}{\circle*{0.15}}
\linethickness{0.3mm}
\put(2,5){\vector(4,-3){4}}
\put(6,5){\vector(1,-3){0.9}}
\put(12,5){\vector(-4,-3){4}}
\put(6.8,1){$X$}
\end{picture}
\eeq
where $p$ is finite in case we have a finite collection of decorations, infinity otherwise.

\begin{logdiffsheaf}
We consider the decoration of points $Z=\{z_1, \cdots, z_N \}$ of $X_N$ by means of chords between their defining points. Let $d_{i-2}$ be the decoration that connects two by two by a chord $i$ defining points of each point of $X_N$. $d_0$ connects two defining points of each point of $X_N$ as seen in \cite{RG}. Over $d_0(X_N)$ we consider log differentials $\dlog \! \vartriangle \!\!z$ whose complex linear span we denote by $\Omega^1 (\log \! \vartriangle \!\! \mathbb{C})$. For a point $Z=\{z_1,\cdots, z_N \}$ for which two of its defining points $z_i$ and $z_j$, $1 \leq i \neq j \leq N$, are connected by a chord in $d_0(X_N)$, we associate the log differential $\dlog (z_i-z_j)$. For three of its defining points being singled out, we consider those decorations that mutually connect all three points together. Let $z_k$ being that third point, other than $z_i$ or $z_j$. We seek a generalization of log differentials corresponding to the triangle with vertices $z_i$, $z_j$ and $z_k$ and edges being chords. We regard such an object as resulting from the coalescing of six distinct points $z^{(a)}$, $1 \leq a \leq 6$, in such a manner that:
\newpage
\begin{itemize}
\item $z^{(1)}, \: z^{(6)} \rightarrow z_i$ \\
\item $z^{(2)}, \: z^{(3)} \rightarrow z_j$ \\
\item $z^{(4)}, \: z^{(5)} \rightarrow z_k$ \\
\end{itemize}
where $z^{(1)}$, $z^{(2)}$ are connected by a chord, as well as $z^{(3)}$, $z^{(4)}$ and $z^{(5)}$, $z^{(6)}$. To each chord corresponds a log differential, and the sum of all three chords corresponds to the sum:
\beq
\dlog(z^{(1)} - z^{(2)}) + \dlog(z^{(3)} - z^{(4)}) + \dlog(z^{(6)} - z^{(5)})
\eeq
which in the limit converges to:
\beq
\dlog(z_i-z_j) + \dlog(z_j-z_k) + \dlog(z_i-z_k)=\dlog(z_i-z_j)(z_j-z_k)(z_k-z_i)
\eeq
and the complex linear span of such elements defines the generalization of $\Omega^1(\log \!\vartriangle \!\! \mathbb{C})$ over $d_1(X_N)$. Continuing in this fashion, if $p \leq N$ is the number of defining points of each point of $X_N$ that are mutually connected by chords, then the complex linear span of elements of the form:
\beq
\dlog \prod_{1 \leq k<l \leq p}(z_{i_k}-z_{i_l})
\eeq
constitutes a generalization of $\Omega^1 (\log \!\vartriangle \!\! \mathbb{C})$ over $d_{p-2}(X_N)$. Thus the collection of generalizations of $\Omega^1(\log \! \vartriangle \!\!\mathbb{C})$ over each $d_{i-2}(X_N)$ constitutes a directed system of sheaves $\Omega^1(\log \! \vartriangle \!\!\mathbb{C})_{d_{i-2}}$, $2 \leq i \leq N$, that we commonly refer to as $\Omega^1(\log \!\vartriangle \!\! \mathbb{C})$. If we denote by $\mathcal{C}$ such a decoration of points by chords, then $(X_N, \mathcal{C}, \Omega^1(\log \!\vartriangle \!\! \mathbb{C}))$ is an emerald over $X_N$ provided we can prove that the generalization of $\Omega^1(\log \!\vartriangle \!\!\mathbb{C})$ over each $d_{i-2}(X_N)$, $2 \leq i \leq N$, is a sheaf. This we do below in Proposition \ref{Omsheaf}.\\

For completeness' sake, $X_N$ is the configuration space of $N$ points in the complex plane, so we can write $X_N(\mathbb{C})$ to emphasize that when we use the word plane we really mean the complex plane. We can generalize this concept of configuration space to any plane. Now if $X_N^{chd}$ denotes the same space for which each of its points has all its defining points connected by chords two by two and $\pi_{d_N}$ is the projection from that space to $X_N$, then in defining $d_{p-2}(X_N)$ we want points of $X_N$ for which only $p$ of their defining points are connected by chords. This amounts to considering $X_{N-p}(\mathbb{C}-\pi_{d_p}X_p^{chd}) \times X_p ^{chd}$ which is the set of pairs $(x,y)$ where $x \in X_{N-p}(\mathbb{C}-\pi_{d_p}(y))$ for $y \in X_p^{chd}$. Then:
\beq
d_{p-2}(X_N)\simeq X_{N-p}(\mathbb{C}-\pi_{d_p}X_p^{chd}) \times X_p ^{chd}
\eeq

As far as the generalization of $\Omega^1(\log \!\vartriangle \!\!\mathbb{C})$ over each $d_{r-2}(X_N)$ is concerned, $2 \leq r \leq N$, for $r$ points connected together by chords, we are interested in elements of the form:
\begin{align}
\dlog \prod_{1 \leq k < l \leq r}(z_{i_k}-z_{i_l})&=\sum_{1 \leq k<l \leq r}\dlog (z_{i_k}-z_{i_l}) \nonumber \\
&=\sum_{1 \leq k<l \leq r}1 \cdot\dlog (z_{i_k}-z_{i_l}) \in \mathbb{P}^{(r+1)(r-2)/2}_{\mathbb{C}} \Omega^1(\log \! \vartriangle \!\! \mathbb{C})
\end{align}
and $\langle \mathbb{P}^{(r+1)(r-2)/2}_{\mathbb{C}} \Omega^1(\log \! \vartriangle \!\! \mathbb{C})\rangle_{\mathbb{C}}$ constitutes the generalization of $\Omega^1 (\log \!\vartriangle \!\! \mathbb{C})$ over $d_{r-2}(X_N)$.\\

\begin{Omsheaf} \label{Omsheaf}
$\Omega^1 (\log \!\vartriangle \!\! \mathbb{C})$ is a sheaf over $d_{p-2}(X_N)$, $2 \leq p \leq N$.
\end{Omsheaf}
\begin{proof}
Fix $p$, $2 \leq p \leq N$. For $U$ open in $X_N$, we first define $\Omega^1(\log \!\vartriangle \!\! \mathbb{C})(U)=\{\sum_{1 \leq i_1, \cdots, i_p \leq N} \lambda_{i_1 \cdots i_p}\dlog \amalg_{1 \leq k<l \leq p}(z_{i_k}-z_{i_l})\,|\,Z=\{z_1,\cdots ,z_N\} \in U ,\;\; \lambda_{i_1 \cdots i_p} \in \mathbb{C}\}$. Let $U=\coprod_{i \in I}V_i$, $V_i$ small enough basic opens in $X_N$, $i \in I$. On each $V_i$ we may have a different local labeling of points from other $V_j$'s, $j \in I,\, j \neq i$. Since the $V_i$'s are disjoint, $\Omega^1(\log \!\vartriangle \!\! \mathbb{C})(U)$ has a consistent set of sections. For $V \subset U$, $V=\coprod_{j \in J} W_j$, where the $W_j$'s are small enough basic opens, for all $j \in J\, \;\exists i \in I$, $W_j \subset V_i$, and we have an obvious restriction map $\rho_{V_i W_j}:\Omega^1(\log \!\vartriangle \!\! \mathbb{C})(V_i) \rightarrow \Omega^1(\log \!\vartriangle \!\! \mathbb{C})(W_j)$. We have $\Omega^1(\log \!\vartriangle \!\! \mathbb{C})(\emptyset)=\emptyset$, and for $U$ any open in $X_N$, $\rho_{UU}=id$. If $W \subseteq V \subseteq U$ are basic open sets with the small enough property, with corresponding restriction maps $\rho_{UV}$, $\rho_{VW}$ and $\rho_{UW}$, then we have $\rho_{VW}\circ \rho_{UV}=\rho_{UW}$. The same results holds for open sets that are not necessarily small enough by the above reasoning. Thus $\Omega^1(\log \!\vartriangle \!\! \mathbb{C})$ is a presheaf. Now let $U$ be open in $X_N$. Let $(U_i)_{i \in I}$ be a covering of $U$ by open sets. We need to show that if $\sigma$ is a section of $\Omega^1(\log \!\vartriangle \!\! \mathbb{C})$ such that $\sigma|_{U_i}=0$ for all i in I, then $\sigma \equiv 0$. This follows from the definition of the sections. Indeed, fix $i \in I$. Let $U_i=\amalg_{j \in J_i}V_{i;j}$ be a covering of $U_i$ by small enough basic opens. Each basic open $V_{i;j}$ is of the form $V_{i;j}=\{V_{i;j1},...,V_{i;jN}\}$ where the $V_{i;jk},\,1 \leq k \leq N$ are non-overlaping open sets in the complex plane. The statement $\sigma_{U_i}=0$ implies $\sigma_{V_{i;j}}=0$ for all $j \in J_i$, and this for all $i \in I$. Now it suffices to write:
\beq
U=\cup_{i \in I}U_i=\cup_{\substack{i \in I \\ j \in J_i}}V_{i;j}
\eeq
and on each $V_{i;j}$ we have $\sigma|_{V_{i;j}}=0$, so $\sigma \equiv 0$. A slight subtlety arises in the argument as different open sets $U_i$, $U_j$ may be unions of small enough opens that intersect. In that case if say $V_{i;k} \cap V_{j;l} \neq \emptyset$, then $\sigma|_{V_{i;k}}=0$ and $\sigma|_{V_{j;j}}=0$ implies $\sigma|_{V_{i;k} \cup V_{j;l}}=0$ where $V_{i;k} \cup V_{j;l}$ is made small enough by using one labeling throughout. Finally if $U$ is open in $X_N$, $(U_i)_{i \in I}$ is an open covering of $U$ and we adopt the same notations as above for the covering of each open $U_i$ by small enough basic opens, then suppose $s_i$ is a section of $\Omega^1(\log \! \vartriangle \!\! \mathbb{C})(U_i)$, $s_j$ is a section of $\Omega^1(\log \!\vartriangle \!\! \mathbb{C})(U_j)$, such that $s_i|_{U_i \cap U_j}=s_j|_{U_i \cap U_j}$, and this for all $i,j$ in $I$, then we want to show there is a global section on $U$. Suppose $U_i \cap U_j \neq \emptyset$, $U_i=\amalg_{k \in J_i}V_{i;k}$, $U_j=\amalg_{l \in J_j}V_{j;l}$, then there is at least some $k \in J_i$ and some $l \in J_j$ such that $V_{i;k} \cap V_{j;l} \neq \emptyset$. $V_{i;k} \cap V_{j;l}$ however is not small enough as we may have two possibly different labelings on the intersection. Thus on $V_{i;k} \cap V_{j;l}$, $s_i$ and $s_j$ must be label independent, and both $V_{i;k}$ and $V_{j;l}$ being small enough, this label independence must extend to the union $V_{i;k} \cup V_{j;l}$. Thus on the union we can write:
\beq
s_i=s_j=\sum \lambda_{i_1 \cdots i_p} \dlog \prod_{1 \leq r < s \leq p}(z_{i_r}-z_{i_s})
\eeq
and this for all non empty intersections $V_{i;k} \cap V_{j;l}$, and this for all intersections $U_i \cap U_j$, which leads to having a globally defined section of $\Omega^1(\log \!\vartriangle \!\! \mathbb{C})$ over $U$. Thus $\Omega^1(\log \!\vartriangle \!\! \mathbb{C})$ is a sheaf over $d_{p-2}(X_N)$, $2 \leq p \leq N$.
\end{proof}
\end{logdiffsheaf}

\section{Relaxed $\check{C}$ech cohomology}
Without pretense to originality, we briefly review elementary $\check{C}$ech cohomology as it is covered in ~\cite{H}. If X is a topological space, $\mathcal{U}=(U_i)_{i \in I}$ an open covering of X, with a fixed well-ordering of the indexing set I, $\mathcal{F}$ a sheaf of abelian groups on X, we define a complex $C^{\cdot}(\mathcal{U};\mathcal{F})$ as follows: for all $p \geq 0$, we define:
\beq
C^p(\mathcal{U};\mathcal{F}):=\prod_{i_0 < \cdots < i_p} \mathcal{F}(U_{i_0} \cap \cdots \cap U_{i_p})
\eeq
To give an element $\psi$ of $C^p(\mathcal{U};\mathcal{F})$ is equivalent to giving a collection $\{\psi_{i_0 \cdots i_p}\;|\; i_0 < \cdots < i_p\}$ where $\psi_{i_0 \cdots i_p} \in \mathcal{F}(U_{i_0} \cap \cdots \cap U_{i_p})$. The coboundary operator $\delta:= \delta_p : C^p \rightarrow C^{p+1}$ is defined as follows on each component:
\beq
(\delta_p \psi)_{i_0 \cdots i_{p+1}}:=\sum_{0 \leq k \leq p+1}(-1)^k \psi_{i_0 \cdots \hat{i}_k \cdots i_{p+1}}|_{U_{i_0} \cap \cdots \cap U_{i_{p+1}}}
\eeq
where the hat on an index indicates that it is omitted. We define the $p$-th $\check{C}$ech cohomology group of $\mathcal{F}$ with respect to the covering $\mathcal{U}$ to be:
\beq
\check{H}^p(\mathcal{U};\mathcal{F}):=h^p(C^{\cdot}(\mathcal{U};\mathcal{F}))
\eeq
We now define a generalized version of this theory that we will call relaxed $\check{C}$ech cohomology, whereby intersections are defined up to equivalence.
\begin{relCech}
Let $d$ be a decoration of $X_N$ defined by an equivalence between the defining points of each point of $X_N$. Let $p \leq N$ be the number of such defining points that are considered to be equivalent under $d$. Let $U=\{U_1, \cdots, U_N\}$ be an open set of $Z=\{z_1, \cdots, z_N\} \in X_N$, $1 \leq i_0 < \cdots < i_{p-1} \leq N$ be $p$ indices for $p$ defining points related to each other by the above equivalence. Then we define:
\beq
U_{i_0} \cap_d \cdots \cap_d U_{i_{p-1}}=\{\text{defining pts in }U_{i_0}, \cdots, U_{i_{p-1}} \text{ resp. that are equivalent}\}
\eeq
Let $\mathcal{F}_d$ be a sheaf over $d(X_N)$ whose definition takes account only of the decoration of defining points. Then we write $\mathcal{F}_d(U_{i_0} \cap_d \cdots \cap_d U_{i_p})$ for $\mathcal{F}_d(U)$ and we define the $p$-th relaxed $\check{C}$ech cohomology groups as:
\beq
\check{C}^p_{rlx}(\mathcal{U};\mathcal{F}_d)=\prod_{1 \leq i_0 < \cdots < i_p \leq N} \mathcal{F}_d(U_{i_0} \cap_d \cdots \cap_d U_{i_p})
\eeq
For $(X_N, \mathcal{D}, \mathcal{F})$ an emerald over $X_N$, $\mathcal{D}=\{d_{p-2} \: |\: 2 \leq p \leq N\}$, $d_{p-2}$ decorations between $p$ defining points, $\mathcal{F}=\{\mathcal{F}_d\}_{d \in \mathcal{D}}$, we write:
\beq
U_{i_0} \cap_{\mathcal{D}} \cdots \cap_{\mathcal{D}} U_{i_{p-1}}=U_{i_0} \cap_{d_{p-2}} \cdots \cap_{d_{p-2}} U_{i_{p-1}} \qquad ; \quad 2 \leq p \leq N
\eeq
and:
\beq
\check{C}^{p-1}_{rlx}(\mathcal{U};\mathcal{F})=\check{C}^{p-1}_{rlx}(\mathcal{U};\mathcal{F}_{d_{p-2}})
\eeq
\end{relCech}

An element $\psi$ of $\check{C}^p_{rlx}(\mathcal{U};\mathcal{F})$ is given by a collection of elements $\psi_{i_0  \cdots  i_p}$ for indices ${i_0 < \cdots < i_p}$. We have:
\beq
(\delta_p \psi)_{i_0, \cdots, i_{p+1}}=\sum_{0 \leq k \leq p+1}(-1)^k \psi_{i_0 \cdot \hat{i}_k \cdots i_{p+1}}\big|_{U_{i_0} \cap_{\mathcal{D}} \cdots \cap_{\mathcal{D}} U_{i_{p+1}}}
\eeq
We define the $p$-th relaxed $\check{C}$ech cohomology groups as:
\beq
\check{H}^p_{rlx}(\mathcal{U};\mathcal{F}):=h^p(\check{C}^{\cdot}_{rlx}(\mathcal{U};\mathcal{F}))
\eeq

\section{Application: the Kontsevich Integral}
For $P$ a pairing of order 1, $|P|=1$, we regard $|P\rangle (Z) \dlog \!\vartriangle \!\!z[P] (Z)$ over $Z \in X_N$ as being equivalent to having a pair $(|P\rangle (Z), \dlog \!\vartriangle \!\!z[P](Z))$ over the same point, or equivalently an element $\dlog \vartriangle \!\!z[P](Z)$ over $|P\rangle (Z) \in d_0(X_N)$ where $d_0 \in \mathcal{C}$. In other terms the element:
\beq
\Omega=\sum_{|P|=1}|P\rangle \dlog \vartriangle \!\!z[P]
\eeq
corresponds to considering all the generators of $\Omega^1(\log \!\vartriangle \!\! \mathbb{C})$ over $d_0(X_N)$. For a point $Z=\{z_1, \cdots, z_N\} \in X_N$, $1 \leq i<j \leq N$, $\dlog (z_i-z_j) \in \Omega^1 (\log \!\vartriangle \!\! \mathbb{C})_Z$ is actually an element of $\Omega^1 (\log \!\vartriangle \!\!\mathbb{C})(U_i \cap_{\mathcal{D}} U_j)$, $U=\{U_1, \cdots, U_N\}$ a neighborhood of $Z$ in $X_N$. Thus:
\beq
\prod_{i<j} \{\dlog(z_i-z_j)\} \in \prod_{i<j} \Omega^1 (\log \!\vartriangle \!\! \mathbb{C})(U_i \cap_{\mathcal{D}} U_j)=\check{C}^1_{rlx} (\mathcal{U};\Omega^1 (\log \!\vartriangle \!\! \mathbb{C}))
\eeq
and on $\check{C}^p_{rlx} (\mathcal{U};\Omega^1 (\log \!\vartriangle \!\! \mathbb{C}))$ we have:
\begin{align}
(\delta_p \psi)_{i_0, \cdots, i_{p+1}}&=\sum_{0 \leq k \leq p+1}(-1)^k \psi_{i_0 \cdot \hat{i}_k \cdots i_{p+1}}\big|_{U_{i_0} \cap_{\mathcal{D}} \cdots \cap_{\mathcal{D}} U_{i_{p+1}}}\\
&=\sum_{0 \leq k \leq p+1}(-1)^k \psi_{i_0 \cdot \hat{i}_k \cdots i_{p+1}}
\end{align}
We are therefore led to regarding $\Omega$ as a one form in relaxed $\check{C}$ech cohomology. In \cite{RG2} we saw that such a form comes up in the expression for some holonomy; if we put a given link $L$ in braid position and denote by $\tilde{\gamma} \in \mathbb{C} \times I$ the geometric braid we obtain, if this latter corresponds to the lift of some loop $\gamma$ in $X_N$, then the holonomy of the KZ connection on $\overline{\mathcal{A}}$ over $X_N$ along $\gamma$ can be used to generate the Kontsevich integral $Z(L)$. If $h$ denotes this holonomy, we write:
\beq
h=e^{\int_I \omega(\gamma)}=\sum_{m\geq 0}\frac{1}{m!}\int_{I^m} \omega^m
\eeq
where $\omega=\Omega/2 \pi i$. We can rewrite this as follows: $\log h$ defines a map:
\begin{align}
\langle\:,\:\rangle: \check{C}^1_{rlx} (\mathcal{U};\Omega^1 (\log \!\vartriangle \!\! \mathbb{C})) \times LX_N & \rightarrow \mathcal{A}^1(\tilde{\gamma})  \nonumber \\
(\omega, \gamma) & \mapsto \langle \omega, \gamma \rangle =\int_I \omega(\tilde{\gamma})
\end{align}
For $\gamma$ fixed:
\begin{align}
\langle \quad, \gamma \rangle: \check{C}^1_{rlx} (\mathcal{U};\Omega^1 (\log \!\vartriangle \!\! \mathbb{C})) & \rightarrow \mathcal{A}^1(\tilde{\gamma}) \nonumber \\
\omega \mapsto \int_I \omega(\tilde{\gamma})
\end{align}
is a homomorphism. This suggests that geometric braids are dual objects to one-forms in relaxed $\check{C}$ech cohomology. We can make this more precise as follows.\\

In this setting, it is not $X_N$ that is more natural to consider but as the above map suggests it is the loop space $LX_N$ of $X_N$. The concept of small enough basic open sets in $X_N$ carries over to $LX_N$. For $\gamma \in LX_N$, $\tilde{\gamma}$ its associated geometric braid in $\mathbb{C} \times I$, a basic open set about $\gamma$ is given by a collection $U=\{U_1, \cdots, U_N\}$ of $N$ disjoint open cylinders about each of the $N$ strands defining $\tilde{\gamma}$, homeomorphic to $\{z \;\; | \;\; |z|<1\} \times I$. Such open sets are said to be small enough if for all $t \in I$, $1 \leq \i \leq N$, $U_i(t)=U_i \cap \mathbb{C} \times \{t\}$ is small enough. Defining points are now replaced by strands, and decorations of points of $LX_N$ are now given by horizontal chords between strands of the geometric braid associated to a given element of $LX_N$. It is not difficult to see that the pull-back of log differentials to the unit interval forms a sheaf over $LX_N$, and as we did above we even have an emerald over $LX_N$, given by the decorations of strands as just discussed, and the pull-back of $\Omega^1 (\log \! \vartriangle \!\! \mathbb{C})$ to $I$ along with its generalizations $\langle \mathbb{P}^{(r+1)(r-2)/2} \Omega^1 (\log \! \vartriangle \!\! \mathbb{C}) \rangle _{\mathbb{C}}$ over each $d_{r-2}(LX_N)$. We denote by $\mathcal{U}$ an opening covering of $LX_N$ by small enough basic open sets. The intersection of open sets up to the decoration provided by homeomorphism types of horizontal chords between cylinders is defined path-wise; two open cylinders are defined to be intersecting up to chords if for all $t \in I$, their respective points in $\mathbb{C} \times \{t\}$ are connected by chords. Observe that since we consider homeomorphism classes of horizontal chords it suffices to have one such $t$ to define this intersection. The relaxed $\check{C}$ech cohomology we are now considering is a path space version of the one we had before and we denote it by $P\check{C}^\bullet _{rlx}(\mathcal{U};\Omega^1(\log \!\vartriangle \!\! \mathbb{C}))$.\\

This definition lends itself well to seeing such a group as the group characterizing deformations of elements of $LX_N$. Indeed deformations of geometric braids can be conveniently represented by pull-backs of forms along their strands to the unit interval $I$. Relative deformations such as those given by log differentials have additional features: not only do they give the relative deformations of strands with respect to one another, they also become singular when one or two strands come together to form a singular point. As a first observation, we have the following fact:
\beq
P\check{C}^1 _{rlx}(\mathcal{U};\Omega^1 (\log \! \vartriangle \!\! \mathbb{C})) \subset T^*LX_N
\eeq
Additionally, since log differentials become singular when strands are brought together, this means we can incorporate singular geometric braids into the picture. We proceed as follows:
\begin{PartLXN}
We denote by $LX_{N,p}$ the set of geometric braids with $p$ strands coming together at a single points of $\mathbb{C} \times I$, in such a manner that each pair of the $p$ strands involved in one given such singular point can be locally deformed to give a transversal intersection. We denote by $\mathbf{LX_N}$ the disjoint union of all such sets, with $LX_{N,0}=LX_N$ as previously considered. We can write:
\beq
\mathbf{LX_N}=LX_N \amalg LX_{N,1} \amalg \cdots LX_{N,N} \nonumber
\eeq
\end{PartLXN}
Note that this definition is somewhat artificial since a singular geometric braid can have singular points resulting from different numbers of strands coming together. It suffices to use concatenation to arrive at such a picture. With the above definition, $LX_{N,p}$ is fully determined by the sheaf $\langle \mathbb{P}^{(p+1)(p-2)/2}  \Omega^1 (\log \! \vartriangle \!\! \mathbb{C}) \rangle_{\mathbb{C}}$ over $d_{p-2}(LX_N)$, or equivalently by $P\check{C}^1 _{rlx}(\mathcal{U};\Omega^1 (\log \! \vartriangle \!\! \mathbb{C}))$. Indeed, a singular point resulting from the collapse of $p$ strands indexed by some indexing set $I_p$ is the singular locus of terms of the form $\sum_{i_k,i_j \in I_p} \lambda_{i_k,i_j} \dlog (z_{i_k}-z_{i_j})$, a simplest such expression being provided by $\prod_{i_k,i_j \in I_p} \dlog (z_{i_k}-z_{i_j})$, a generator of $\langle \mathbb{P}^{(p+1)(p-2)/2}  \Omega^1 (\log \! \vartriangle \!\! \mathbb{C}) \rangle_{\mathbb{C}}$, a sheaf over $d_{p-2}(LX_N)$. Away from those singular points the geometric braids are fully determined by the knowledge of log differentials restricted to the non-singular part of such braids, and this is given by the same sheaf. More generally, $\mathbf{LX_N}$ is fully determined by $P\check{C}^\bullet _{rlx}(\mathcal{U}; \Omega^1(\log \! \vartriangle \!\! \mathbb{C})$.\\

Finally we can give a pairing between loops in $X_N$ and relaxed $\check{C}$ech forms without having to resort to using an artificial integration over $I$. Consider the same pairing as introduced before:
\begin{align}
LX_N \times P \check{C}^1 _{rlx}(\mathcal{U}; \Omega^1 (\log \! \vartriangle \!\! \mathbb{C})) &\rightarrow \mathcal{A}^1 \nonumber \\
(\gamma, \omega) &\mapsto \int_I \omega (\tilde{\gamma})=\int_I \gamma^* \omega = \int_{\tilde{\gamma}}\omega
\end{align}
In this manner we have a map reminiscent of a cup product, thereby presenting braids as relaxed $\check{C}$ech 1-chains.

\end{document}